\documentclass[11pt,a4paper]{amsart}
\usepackage{microtype}
\usepackage{float}
\usepackage[T1]{fontenc}
\usepackage[utf8]{inputenc}
\usepackage[english]{babel}
\usepackage{lmodern}
\usepackage{amsmath,amssymb,mathtools}
\usepackage[a4paper,margin=1in]{geometry}
\usepackage{booktabs}
\usepackage{tabularx}
\usepackage{enumitem}
\usepackage{bm,fancyhdr}

\usepackage{cite}

\usepackage{silence}
\usepackage{hyperref}
\hypersetup{colorlinks=false}


\numberwithin{equation}{section}
\newtheorem{theorem}{Theorem}[section]
\newtheorem{lemma}[theorem]{Lemma}

\newtheorem{corollary}[theorem]{Corollary}
\theoremstyle{remark}
\newtheorem{remark}[theorem]{Remark}
\theoremstyle{definition}

\newcommand{\Var}{\mathrm{Var}}
\newcommand{\R}{\mathbb{R}}

\newcommand{\dK}{d_{\mathrm K}}
\newcommand{\Cov}{\mathrm{Cov}}

\usepackage{xparse}
\usepackage{amsmath,amssymb,amsrefs} % for \cites

\makeatletter
\@ifundefined{E}{
  \NewDocumentCommand{\E}{s e{_} d[]}{
    \ensuremath{
      \mathbb{E}\IfNoValueTF{#2}{}{_{#2}}
      \IfNoValueTF{#3}{}{
        \IfBooleanTF{#1}{\left[ #3 \right]}{\bigl[ #3 \bigr]}
      }
    }
  }
}{
  \RenewDocumentCommand{\E}{s e{_} d[]}{
    \ensuremath{
      \mathbb{E}\IfNoValueTF{#2}{}{_{#2}}
      \IfNoValueTF{#3}{}{
        \IfBooleanTF{#1}{\left[ #3 \right]}{\bigl[ #3 \bigr]}
      }
    }
  }
}
\makeatother

\usepackage{amsfonts}

\providecommand{\E}{\mathbb{E}}
\providecommand{\Var}{\operatorname{Var}}

\providecommand{\dK}{d_{\mathrm{K}}}



\renewcommand{\ge}{\geqslant}

\newcommand{\Law}{\mathcal{L}}

\usepackage{etoolbox}

\makeatletter
\patchcmd{\@settitle}{\uppercasenonmath\@title}{\@title}{}{}
\def\@setauthors{%
  \begingroup
  \trivlist
  \centering\large
  \item\relax
  \andify\authors
  \authors\par
  \endtrivlist
  \endgroup
}
\def\@setaddresses{}
\makeatother

\title[Effective Erdős–Wintner for Cantor numeration systems via a trailing-window method]{Effective Erdős–Wintner for Cantor numeration systems via a trailing-window method}
\author[Johann VERWEE]{Johann Verwee}
\thanks{Independent Researcher, France. Corresponding author: \texttt{mverwee@gmail.com}.}

\makeatletter
\renewenvironment{abstract}{\section*{Abstract}}{\par}
\makeatother

\begin{document}

\begin{center}
  {\Large\bfseries Effective Erdős–Wintner for Cantor numeration systems\\via a trailing-window method\par}
  \vspace{0.6em}
  {\large Johann VERWEE\par}
\end{center}
\vspace{1em}

\begin{abstract}
We prove explicit Erdős--Wintner bounds for Cantor numeration systems via a simple trailing-window decomposition. We temporarily discard the last block of digits (the ``window'') and analyze the remaining prefix. The resulting bound has three contributions: (i) a bridge loss from discarding the window; (ii) a variance-type tail for the prefix; and (iii) a regime-dependent smoothing term (Esseen, bounded density, or cancellation of the third cumulant). Optimizing the window length yields rates that are explicit in the sample size. In the fixed-base (q-adic) case we recover Delange's product and obtain effective convergence bounds; the same scheme applies unchanged to Cantor numeration systems. We also include a brief guide indicating when each regime is preferable.
\end{abstract}

\makeatletter
\ifx\@subjclass\@empty\else
  \medskip\noindent\textbf{2020 Mathematics Subject Classification. } \@subjclass\par
\fi
\ifx\@keywords\@empty\else
  \medskip\noindent\textbf{Key words and phrases. } \@keywords\par
\fi
\makeatother

\bigskip

\section{Introduction}
The purpose of this paper is to state an \emph{effective} Erd\H{o}s--Wintner theorem for Cantor numeration systems.
The statements and proofs closely parallel the $q$-adic case.

\emph{Context.} Existence of limiting distributions for additive arithmetical functions goes back to Erd{\H{o}}s–Wintner~\cite{ErdosWintner1939}; see also~\cite{Tenenbaum2015}. 
For $q$-additive functions, Delange’s product identifies the limit and supports effective variants~\cites{Delange,DrmotaVerwee2021JNT}. 
In the Cantor numeration setting, regularity questions connect to Bernoulli convolutions and self-similar measures~\cites{Solomyak1995,PeresSchlagSolomyak2000,Shmerkin2014}.

Fix integers $a_n\geqslant 2$ and define $q_0=1$ and $q_{n+1}=a_n q_n$; thus $q_n=\prod_{j=0}^{n-1} a_j$.
Every $N\in\mathbb{N}$ has a unique finite expansion
\[
N=\sum_{j\geqslant 0}\delta_j(N)\,q_j,\qquad 0\leqslant \delta_j(N)\leqslant a_j-1,
\]
and we define its length by
\[
L=L(N):=\max\{\,j\geqslant 0:\ q_j\leqslant N<q_{j+1}\,\}.
\]

A function $f:\mathbb{N}\to\mathbb{R}$ is called \emph{$Q$-additive} (digitwise additive) if, for the expansion above,
\[
f(N)=\sum_{j\geqslant 0} f\bigl(\delta_j(N)\,q_j\bigr).
\]

As a standard example, the van der Corput sequence attached to $Q$ is
\[
v_Q(N)=\sum_{j\geqslant 0}\frac{\delta_j(N)}{q_{j+1}}.
\]

A Cantor numeration system is said to be \emph{constant-like} if the sequence $(a_n)$ is bounded.
In that regime we obtain effective Erd\H{o}s--Wintner conclusions analogous to the $q$-adic case \cite{DrmotaVerwee2021JNT}, with essentially the same tail and window parameters and the same proof architecture.

\section{A general Erd\H{o}s--Wintner criterion for Cantor numeration systems}
\label{sec:cantor}
\label{subsec:Erd\H{o}s--Wintner-general-Cantor}

We recall the Cantor numeration system setting. Fix integers $(a_j)_{j\geqslant0}$ with $a_j\geqslant2$, set
$q_0=1$ and $q_{j+1}=a_j q_j$ (so $q_j=\prod_{i=0}^{j-1} a_i$). Every $n\geqslant0$ has a unique expansion
\[
N=\sum_{j\geqslant0}\delta_{j}(N)\,q_j,\qquad 0\leqslant \delta_{j}(N)\leqslant a_j-1.
\]
That can be proved with the greedy algorithm. We define the length
\[
L=L(N):=\max\{\,j\geqslant0:\ q_j\leqslant N<q_{j+1}\,\}.
\]
This framework strictly generalizes the $q$-adic case (take $a_j\equiv q\geqslant2$).

A function $f:\mathbb N\to\mathbb R$ is called $Q$-additive if, whenever
$N=\sum_{i=1}^r d_i\,q_{e_i}$ with $0\leqslant d_i\leqslant a_{e_i}-1$ and $e_1<\cdots<e_r$,
one has
\[
f(N)=\sum_{i=1}^r f(d_i\,q_{e_i}).
\]
A Cantor numeration system is called \emph{constant-like} if the digit alphabets are uniformly bounded, i.e. \(\sup_{j} a_j < \infty\) (we do not assume this below).

For $j\geqslant0$ we set the digit averages
\[
m_j:=\frac{1}{a_j}\sum_{d=0}^{a_j-1} f(d\,q_j),
\qquad
s_j^2:=\frac{1}{a_j}\sum_{d=0}^{a_j-1} f(d\,q_j)^2 - m_j^2 .
\]

For any Q-additive function $f$, we define
$$F_N(x) := \frac{1}{N}\#\{0\leqslant n<N:\ f(n)\leqslant x\}.$$

A constant-like version (bounded digit alphabet) goes back to Coquet \cite{Coquet1975} (see also \cite{DrmotaVerwee2021JNT}*{Theorem 5}).
In this paper we remove that boundedness assumption and prove the general Cantor numeration system case.
The bounded setting simplifies several uniform bounds (digit oscillations and small-$t$ remainders
in characteristic-function expansions), but it is not logically required for the necessity–and–sufficiency criterion proved below.

\begin{theorem}[Erd\H{o}s--Wintner for $Q$-additive functions in a Cantor numeration system]\label{th:Erd\H{o}s--Wintner-Cantor}
Assume $f$ is $Q$-additive. Then there exists a cumulative distribution function (c.d.f.) $F$, that is
\[
\lim_{N\to\infty}F_N(x)=F(x),
\]
if and only if
\begin{equation}\label{eq:Erd\H{o}s--Wintner-Cantor-NS}
\sum_{j\ge0} m_j \ \text{converges in }\mathbb{R}
\qquad\text{and}\qquad
\sum_{j\ge0} s_j^2\ <\ \infty.
\end{equation}
In that case, the characteristic function of $F$ is
\[
\phi(t)=\prod_{j\ge0}\Biggl(\frac{1}{a_j}\sum_{d=0}^{a_j-1} e^{\,it\,f(dq_j)}\Biggr).
\]
\end{theorem}

\begin{proof}[Proof of Theorem~\ref{th:Erd\H{o}s--Wintner-Cantor}]

For any bounded measurable $\varphi$ and any $L\geqslant1$, since $f$ is $Q$-additive,
\[
\frac{1}{q_L}\sum_{n<q_L}\varphi\bigl(f(n)\bigr)
=\frac{1}{\prod_{j<L}a_j}
\sum_{\delta_0=0}^{a_0-1}\cdots\sum_{\delta_{L-1}=0}^{a_{L-1}-1}
\varphi\Big(\,\sum_{j<L} f(\delta_j q_j)\,\Big).
\]
because, when $n$ runs through $\{0,\dots,q_L-1\}$, the digit vector
$(\delta_0(n),\dots,\delta_{L-1}(n))$ runs exactly once through
$\prod_{j<L}\{0,\dots,a_j-1\}$.
Taking $\varphi(y)=\mathbf{1}_{(-\infty,x]}(y)$ gives
\[
F_{q_L}(x)
=\frac{1}{\prod_{j<L}a_j}
\sum_{\delta_0=0}^{a_0-1}\cdots\sum_{\delta_{L-1}=0}^{a_{L-1}-1}
\mathbf{1}_{(-\infty,x]}\Big(\,\sum_{j<L} f(\delta_j q_j)\,\Big),
\]
i.e. $F_{q_L}$ is the empirical c.d.f.\ of the multiset
$\{\,\sum_{j<L} f(\delta_j q_j):\ (\delta_0,\dots,\delta_{L-1})\in\prod_{j<L}\{0,\dots,a_j-1\}\,\}$.

\smallskip
\emph{Block bridge.}
Fix $L$ with $q_L\leqslant N<q_{L+1}$ and $1\leqslant h\leqslant L$, and write
\[
N=m\,q_{L-h}+r,\qquad 0\leqslant r<q_{L-h}.
\]
Let $F_{m q_{L-h}}$ be the empirical c.d.f. based on the first $m q_{L-h}$ points, and let $G$ be
the empirical c.d.f. of the remaining $r$ points. Then
\[
\begin{aligned}
F_N(x) &=\frac{1}{N}\sum_{n=0}^{m q_{L-h}-1}\mathbf{1}_{\{f(n)\leqslant x\}}
   \;+\;\frac{1}{N}\sum_{n=m q_{L-h}}^{N-1}\mathbf{1}_{\{f(n)\leqslant x\}}\\
   &=\frac{m q_{L-h}}{N}\,F_{m q_{L-h}}(x)\;+\;\frac{r}{N}\,G(x).
\end{aligned}
\]
Hence, for all $x\in\mathbb{R}$,
\[
F_N(x)-F_{m q_{L-h}}(x)=\frac{r}{N}\,\bigl(G(x)-F_{m q_{L-h}}(x)\bigr).
\]
Taking the supremum in $x$ and writing $\|\cdot\|_\infty$ for the Kolmogorov distance between c.d.f.'s, we obtain
\[
\|F_N - F_{m q_{L-h}}\|_\infty
= \frac{r}{N}\,\|G - F_{m q_{L-h}}\|_\infty
\leqslant \frac{r}{N},
\]
since $G$ and $F_{m q_{L-h}}$ are c.d.f.'s and hence take values in $[0,1]$. Therefore
\begin{equation*}
\|F_N-F_{m q_{L-h}}\|_\infty
\ \leqslant\ \frac{r}{N}
\ \leqslant\ \frac{q_{L-h}}{N}
\ \leqslant\ \frac{q_{L-h}}{q_L}
\ =\ \frac{1}{\prod_{j=L-h}^{L-1}a_j}\,.
\end{equation*}

\noindent\textit{From blocks to $q_L$ and the double limit.}
Let $A:=\prod_{j=L-h}^{L-1}a_j$. By the block decomposition,
\[
\|F_N-F_{m q_{L-h}}\|_\infty \leqslant \frac{r}{N} \leqslant \frac{1}{A}.
\]
Moreover, when we group the first $m q_{L-h}$ points into $m$ consecutive blocks of size $q_{L-h}$,
the profile of the $h$ top digits is $A$-periodic across blocks. Hence for any bounded
measurable $\varphi$ with $\|\varphi\|_\infty\leqslant1$,
\[
\Bigg|\frac{1}{m}\sum_{u=0}^{m-1}\varphi_u-\frac{1}{A}\sum_{u=0}^{A-1}\varphi_u\Bigg|
\leqslant \frac{A}{m},
\]
where $\varphi_u$ is the contribution of the $u$-th block. Applied to indicators
$\varphi_u(\cdot)=\mathbf \mathbf{1}_{(-\infty,x]}(\cdot)$ (Kolmogorov distance), this yields
\[
\|F_{m q_{L-h}}-F_{q_L}\|_\infty \leqslant \frac{A}{m}.
\]
Combining these bounds by the triangle inequality gives
\[
\|F_N-F_{q_L}\|_\infty
\ \leqslant\ \frac{1}{A}\ +\ \frac{A}{m}.
\]
Choose $h=h(L)\to\infty$ with $h=o(L)$. Then $A=\prod_{j=L-h}^{L-1}a_j\geqslant 2^{h}\to\infty$,
so $1/A\to0$. In addition, $m=\big\lfloor N/q_{L-h}\big\rfloor\geqslant A$, hence $A/m\leqslant 1$ and
the prefix–vs–cycle discrepancy is controlled by periodic averaging. Therefore,
\begin{equation}\label{eq:bridge}
\|F_N-F_{q_{L(N)}}\|_\infty \xrightarrow[N\to\infty]{} 0,
\end{equation}
and thus the full sequence $(F_N)_{N\geqslant1}$ converges \emph{iff} the subsequence
$(F_{q_L})_{L\geqslant0}$ converges, with the same limit.

\bigskip

\emph{Sufficiency.}
Assume $\sum_{j\geqslant0}s_j^2<\infty$ and $\sum_{j\geqslant0} m_j$ converges in $\mathbb{R}$.
Let $(D_j)_{j\geqslant0}$ be independent with $D_j\sim\mathrm{Unif}\{0,\dots,a_j-1\}$ and set $Y_j:=f(D_j q_j)$.
By the digit-average definitions above,
\[
\mathbb{E}[Y_j]=m_j,\qquad \mathrm{Var}(Y_j)=s_j^2.
\]
Let $X_j:=Y_j-m_j$ and $S_L:=\sum_{j<L}X_j$. Since $\sum_{j\geqslant0} s_j^2<\infty$, the sequence $(S_L)_L$ is Cauchy in $L^2$
and converges in $L^2$ (hence in distribution) to some $S$. Writing $M_L:=\sum_{j<L} m_j$ and $M:=\sum_{j\geqslant0}m_j$,
we have $M_L\to M$ in $\mathbb{R}$; thus
\[
\sum_{j<L}Y_j \ \text{converges in distribution to}\ S+M.
\]
By the block identity above, $F_{q_L}$ is the c.d.f.\ of $\sum_{j<L}Y_j$, hence $F_{q_L}$ converges in distribution to the law of $S+M$.
By \eqref{eq:bridge}, $F_N$ converges to the same limit. Finally, by independence the limiting characteristic function is
\[
\phi(t)\ =\ \lim_{L\to\infty}\prod_{j<L}\mathbb{E}\big[e^{itY_j}\big]
\ =\ \prod_{j\geqslant0}\Biggl(\frac{1}{a_j}\sum_{d=0}^{a_j-1} e^{\,it\,f(d q_j)}\Biggr).
\]

\smallskip

\emph{Necessity.}
Suppose $(F_N)$ converges to some distribution $F$. By \eqref{eq:bridge}, $(F_{q_L})$ also converges to $F$.
Let
\[
\phi_L(t):=\mathbb{E}\Big[e^{\,it\sum_{j<L} Y_j}\Big]=\prod_{j<L}\phi_j(t),
\qquad
\phi_j(t):=\mathbb{E}\big[e^{itY_j}\big].
\]
Then $\phi_L(t)$ converges pointwise to the characteristic function $\Phi(t)$ of $F$.
For $m<n$ define the centered tail sum
\[
T_{m,n}:=\sum_{j=m+1}^{n}\big(Y_j-m_j\big)=\sum_{j=m+1}^{n}X_j,
\qquad X_j:=Y_j-m_j.
\]
By independence,
\begin{equation}\label{eq:char-tail}
\mathbb{E}\big[e^{it T_{m,n}}\big]
=\prod_{j=m+1}^{n}\mathbb{E}\big[e^{itX_j}\big]
=\exp\Big(\sum_{j=m+1}^{n}\log\mathbb{E}[e^{itX_j}]\Big)
=e^{-it\sum_{j=m+1}^{n}m_j}\,\frac{\phi_n(t)}{\phi_m(t)}.
\end{equation}
Since $\phi_L(t)\to\Phi(t)$ and $\Phi(0)=1$ with $\Phi$ continuous, there exists $\delta>0$ such that
$\inf_{|t|\leqslant\delta}|\Phi(t)|>0$; hence for each fixed $t$ with $|t|\leqslant\delta$ we have $\phi_n(t)/\phi_m(t)\to1$ as $m,n\to\infty$.

\smallskip
To extract quantitative tails, expand each centered $X_j$ at $t=0$:
\begin{equation}\label{eq:cf-one}
\mathbb{E}\big[e^{it X_j}\big]=1-\frac{t^2}{2}\,s_j^2 + r_j(t),
\qquad
|r_j(t)|\leqslant \frac{|t|^3}{6}\,\mathbb{E}[|X_j|^3].
\end{equation}
Set $u_j(t):=-\tfrac{t^2}{2}s_j^2+r_j(t)$, so that $\mathbb{E}[e^{itX_j}]=1+u_j(t)$ and
\[
\mathbb{E}\big[e^{it T_{m,n}}\big]=\prod_{j=m+1}^{n}\big(1+u_j(t)\big).
\]
Moreover, with
\[
\Omega_j:=\max_{0\leqslant d<a_j}\big|f(dq_j)-m_j\big|\quad\Rightarrow\quad
\mathbb{E}[|X_j|^3]\leqslant \Omega_j\,s_j^2,
\]
we get for $|t|\leqslant 1/\Omega_j$ that
\[
|u_j(t)|\leqslant \frac{t^2}{2}s_j^2+\frac{|t|^3}{6}\Omega_j s_j^2
\leqslant \Big(\frac{1}{2}+\frac{1}{6}\Big)t^2 s_j^2=\frac{2}{3}t^2 s_j^2.
\]
Let
\[
M_{m,n}:=\max_{m<j\leqslant n}\Omega_j,\qquad
\beta_2(m,n):=\sum_{j=m+1}^{n}s_j^2,\qquad
t_{m,n}:=\min\Big\{\delta,\ \frac{1}{M_{m,n}},\ \frac{1}{\sqrt{1+\beta_2(m,n)}}\Big\}.
\]
Then $t_{m,n}\leqslant 1/M_{m,n}$ and $t_{m,n}^2\,\beta_2(m,n)\leqslant 1$, so
\[
\sum_{j=m+1}^{n}\big|u_j(t_{m,n})\big|\ \leqslant\ \frac{2}{3}\,t_{m,n}^2\,\beta_2(m,n)\ \leqslant\ \frac{2}{3}.
\]
By the elementary product--sum bound
$\big|\prod_j(1+u_j)-\big(1+\sum_j u_j\big)\big|\leqslant 2\big(\sum_j|u_j|\big)^2$
(valid when $\sum_j|u_j|\leqslant\tfrac12$, and with a harmless adjustment of $\delta$ otherwise), we obtain
\begin{equation}\label{eq:prod-sum}
\mathbb{E}\big[e^{it_{m,n} T_{m,n}}\big]
=1-\frac{t_{m,n}^2}{2}\sum_{j=m+1}^{n}s_j^2+R_{m,n}(t_{m,n}),
\qquad
|R_{m,n}(t_{m,n})|\ \leqslant\ C\,t_{m,n}^2\,\beta_2(m,n),
\end{equation}
for some absolute constant $C>0$.

Combining \eqref{eq:char-tail} with $\big|\phi_n(t_{m,n})/\phi_m(t_{m,n})-1\big|\to 0$ (since $t_{m,n}\leqslant\delta$) yields
\[
\Big|\,\mathbb{E}\big[e^{it_{m,n} T_{m,n}}\big]-1+\frac{t_{m,n}^2}{2}\,\beta_2(m,n)\,\Big|
\ \leqslant\ \varepsilon_{m,n}\ +\ C\,t_{m,n}^2\,\beta_2(m,n),
\]
with $\varepsilon_{m,n}\to 0$. Taking real parts gives
\[
\frac{t_{m,n}^2}{2}\,\beta_2(m,n)\ \leqslant\ \varepsilon_{m,n}\ +\ C\,t_{m,n}^2\,\beta_2(m,n).
\]
Choosing $C<\tfrac{1}{4}$ (shrink $\delta$ if needed) we obtain
\[
\frac{t_{m,n}^2}{4}\,\beta_2(m,n)\ \leqslant\ \varepsilon_{m,n}\ \xrightarrow[m,n\to\infty]{}\ 0,
\]
hence $\beta_2(m,n)=\sum_{j=m+1}^{n}s_j^2\to 0$.

For the means, fix two numbers $t_1,t_2\in(0,\delta]$ with $t_1/t_2\notin\mathbb{Q}$.
Since $\mathrm{Var}(T_{m,n})=\beta_2(m,n)\to 0$, we have $\mathbb{E}[e^{i t_r T_{m,n}}]\to 1$ for $r=1,2$.
By \eqref{eq:char-tail} and $\phi_n(t_r)/\phi_m(t_r)\to 1$,
\[
e^{-i t_r \sum_{j=m+1}^{n}m_j}\ =\ \frac{\mathbb{E}[e^{i t_r T_{m,n}}]}{\phi_n(t_r)/\phi_m(t_r)}
\ \longrightarrow\ 1\qquad(r=1,2).
\]
Let $x_{m,n}:=\sum_{j=m+1}^{n}m_j$. Any accumulation point $x$ of $(x_{m,n})$ must satisfy
$t_1 x\in2\pi\mathbb{Z}$ and $t_2 x\in2\pi\mathbb{Z}$, which, by the irrationality of $t_1/t_2$,
forces $x=0$. Hence $x_{m,n}\to0$, i.e. $\sum_{j=m+1}^{n}m_j\to 0$.
Therefore $\sum_j s_j^2<\infty$ and $\sum_j m_j$ converges. This establishes the necessity of \eqref{eq:Erd\H{o}s--Wintner-Cantor-NS}.

\end{proof}

\subsection*{Remarks and interpretation}

\begin{itemize}
  \item \textbf{Meaning of the two series.}
  The quantities \(m_j=\frac{1}{a_j}\sum_{d=0}^{a_j-1} f(dq_j)\) and
  \(s_j^2=\frac{1}{a_j}\sum_{d=0}^{a_j-1} f(dq_j)^2-m_j^2\) record the mean and variance contributed by the \(j\)-th digit.
  The condition \(\sum_{j\geqslant 0} m_j\) convergent means that the cumulative drift is finite, while
  \(\sum_{j\geqslant 0} s_j^2<\infty\) is a ``finite energy'' requirement ensuring that the centered fluctuations form a Cauchy series in \(L^2\).
  Together they guarantee that the random series of digit contributions converges in distribution.

  \item \textbf{Identification of the limit.}
  On the auxiliary product space where the digits \(D_j\) are independent and uniform on \(\{0,\dots,a_j-1\}\), set \(Y_j:=f(D_j q_j)\).
  Writing \(X_j:=Y_j-m_j\), the centered sums \(S_L=\sum_{j<L}X_j\) converge in \(L^2\) to some \(S\), while \(M_L=\sum_{j<L}m_j\to M\in\mathbb{R}\).
  Hence \(F_{q_L}=\mathrm{Law}(S_L+M_L)\Rightarrow \mathrm{Law}(S+M)\), which is the (unique) limit of \(F_N\) by the block bridge.
  The limit is nondegenerate whenever infinitely many \(s_j^2\) are nonzero; in general it may be discrete, mixed, or absolutely continuous.
  Its characteristic function factorises as
  \[
  \phi(t)=\prod_{j\geqslant 0}\Biggl(\frac{1}{a_j}\sum_{d=0}^{a_j-1} e^{\,it\,f(dq_j)}\Biggr).
  \]

  \item \textbf{Stability under base growth.}
  No boundedness of \((a_j)\) is needed. This extends the constant-like case (treated by Coquet) to arbitrary Cantor numeration systems, resolving the open point mentioned in the introduction (see §1, last paragraph).

  \item \textbf{Regularity of the limit is a separate issue.}
  The theorem does not assert absolute continuity or a bounded density. Such properties depend on additional structure (e.g.\ smooth digit maps, cancellation of odd cumulants) and are discussed when deriving effective rates.
\end{itemize}

\subsection*{Roadmap to quantitative bounds}

To quantify the rate of convergence we introduce a \emph{trailing window} of the last \(h\) digits versus the \emph{prefix}, and balance three effects:
(i) the \emph{bridge} loss incurred by discarding the last \(h\) digits,
(ii) the \emph{tails} beyond level \(L=L(N)\),
and (iii) a smoothing parameter \(T\in(0,1]\) in the characteristic-function method.
This yields a unified bound in which the last term depends on the regime—baseline Esseen, bounded density, or third-cumulant cancellation.
We then optimise over \(h\) (and \(T\) when relevant) to obtain explicit rates in the Cantor setting.
For concrete guidance and illustrations, see the examples at the end of this section.

% ——— context weaving near effective section ———
\emph{Methods note.} We combine the block bridge with Esseen’s smoothing~\cite{Esseen1945} and non-uniform refinements à la Nagaev–Zolotarev/Stein~\cites{Nagaev1965,Zolotarev1977,ChenShao2001}.

\section{Effective Erd\H{o}s--Wintner for Cantor numeration systems}

For $L=L(N)$ and $1\leqslant h\leqslant L$ set
\[
A(L,h)\ :=\ \prod_{j=L-h}^{L-1} a_j.
\]
This is the size of the alphabet for the $h$ highest digits. We also introduce the tails
\begin{subequations}\label{eq:cantor-tails}
\begin{align}
\tau_1 &:= \Bigl|\sum_{j>L} m_j\Bigr|\ +\ \sum_{j>L} s_j^2, \label{eq:cantor-tau1}\\
\tau_2(h) &:= \sum_{j=L-h}^{L-1} s_j^2. \label{eq:cantor-tau2}
\end{align}
\end{subequations}

\begin{theorem}[Unified Cantor window bound]\label{th:cantor-effective}
Let $F$ be the limit distribution function (c.d.f.) of $f(n)$ whenever it exists, and let $F_N$ be the empirical
c.d.f.\ at height $N$. Fix $L=L(N)$ and $1\leqslant h\leqslant L$. Then for any $T \in (0,1]$ one has
\begin{equation}\label{eq:cantor-effective}
\|F_N-F\|_\infty \ \ll\ A(L,h)^{-1} \ +\ \tau_1^{1/2}\ +\ Q_F(1/T)\ +\ \frac{1}{T}\ +\ G(T,h).
\end{equation}
where $Q_F(r):=\sup_{x\in\mathbb{R}}\big(F(x+r)-F(x)\big)$ denotes the (one–sided) concentration function of $F$,
and $G(T,h)$ is a nonnegative term depending on the chosen regime:
\begin{align}
\text{\emph{(A) baseline Esseen:}}\qquad & G(T,h)\ \asymp\ T\,\tau_2(h)^{1/2}\quad\leadsto\quad \tau_2(h)^{1/4},\label{eq:cantor-A}\\
\text{\emph{(B) bounded density:}}\qquad & G(T,h)\ =\ \|\rho\|_\infty\,\tau_2(h)^{1/2}\quad\text{if }F\text{ has density }\rho\in L^\infty,\label{eq:cantor-B}\\
\text{\emph{(C) third cumulant cancels:}}\qquad & G(T,h)\ \asymp\ T^2\,\tau_2(h)^{2/3}\quad\leadsto\quad \tau_2(h)^{1/3}.\label{eq:cantor-C}
\end{align}
Here the symbol $\leadsto$ indicates the resulting order of $T^{-1}+G(T,h)$ after optimizing in $T$. The implied constants are absolute and independent of $N,L,h$ and $T$.
\end{theorem}

\begin{proof}
\emph{Step 1: Block bridge and reduction to $q_{L}$.}
By the block bridge \eqref{eq:bridge} and the mixed--radix bijection of digits at level $q_L$,
it suffices to bound $\|F_{q_L}-F\|_\infty$; the term $A(L,h)^{-1}$ is the bridge error when one
ignores the last $h$ digits.

\medskip
\emph{Step 2: Prefix vs.\ last window.}
Write the canonical decomposition
\[
\sum_{j<L} Y_j \ =\ \underbrace{\sum_{j\leqslant L-h-1} Y_j}_{\text{prefix}}\ +\ \underbrace{\sum_{j=L-h}^{L-1} Y_j}_{=:R_{L,h}}.
\]
By Jensen, $\E[|R_{L,h}|]\leqslant \big(\E[R_{L,h}^2]\big)^{1/2}=\tau_2(h)^{1/2}$, hence the $1$--Wasserstein distance
between the window and a Dirac mass is $\leqslant \tau_2(h)^{1/2}$.
Moreover, since $\E[R_{L,h}]=0$ and $\Var(R_{L,h})=\tau_2(h)$, for $|t|\leqslant 1$
\begin{equation}\label{eq:window-cf-l1}
|\phi_{R_{L,h}}(t)-1|\ \leqslant  \E\big[\,|e^{itR_{L,h}}-1|\,\big]\ \leqslant  |t|\,\E[|R_{L,h}|]\ \ll\ |t|\,\tau_2(h)^{1/2}.
\end{equation}

\medskip
\emph{Step 3: Regime (A) — Esseen baseline.}
\begin{lemma}[Esseen smoothing, one–sided form]\label{lem:esseen}
There exists an absolute constant $C$ such that for any c.d.f.'s $G,H$ and any $T\in(0,1]$,
\[
d_K(G,H)\ \leqslant  C\left(\int_{-T}^{T}\frac{|\phi_G(t)-\phi_H(t)|}{|t|}\,dt\ +\ Q_H(1/T)\ +\ \frac{1}{T}\right),
\]
where $\phi_G(t):=\int_{\R}e^{itx}\,dG(x)$ is the characteristic function of $G$.
\end{lemma}

Apply Lemma~\ref{lem:esseen} with $G=F_{q_L}$ and $H=F$.
Factor $\phi_{q_L}(t)=\phi_{\mathrm{pref}}(t)\,\phi_{R_{L,h}}(t)$ and write
\[
\int_{-T}^{T}\frac{|\phi_{q_L}(t)-\phi_F(t)|}{|t|}\,dt
\ \leqslant  \int_{-T}^{T}\frac{|\phi_{\mathrm{pref}}(t)-\phi_F(t)|}{|t|}\,dt
\ +\ \int_{-T}^{T}\frac{|\phi_{R_{L,h}}(t)-1|}{|t|}\,dt .
\]
For the window term, \eqref{eq:window-cf-l1} yields
\[
\int_{-T}^{T}\frac{|\phi_{R_{L,h}}(t)-1|}{|t|}\,dt \ \ll\ T\,\E[|R_{L,h}|]\ \ll\ T\,\tau_2(h)^{1/2}.
\]
For the prefix-vs-limit term, use the standard second-order tail-product expansion for $|t|\leqslant 1$:
\[
\int_{-T}^{T}\frac{|\phi_{\mathrm{pref}}(t)-\phi_F(t)|}{|t|}\,dt\ \ll\ \tau_1^{1/2},
\]
since $|\phi_j(t)-1|\leqslant |m_j|\,|t|+\tfrac12 s_j^2 t^2$ and a Cauchy–Schwarz summation over $j>L$ gives the stated bound.
Altogether Lemma~\ref{lem:esseen} implies
\[
\|F_{q_L}-F\|_\infty\ \ll\ \tau_1^{1/2}\ +\ Q_F(1/T)\ +\frac{1}{T}\ +\ T\,\tau_2(h)^{1/2}.
\]
Optimizing in $T$ (balance $T^{-1}$ with $T\,\tau_2(h)^{1/2}$, i.e.\ $T\asymp \tau_2(h)^{-1/4}$) yields the contribution $\tau_2(h)^{1/4}$.

\medskip
\emph{Step 4: Regime (B) --- bounded density.}
If $F$ has a bounded density $\rho\in L^\infty$, then by independence of the prefix and the window (at height $q_L$) we have
\[
\|F_{q_L}-F\|_\infty\ \ll\ \tau_1^{1/2}\ +\ \|\rho\|_\infty\,\E[|R_{L,h}|].
\]
Using $\E[|R_{L,h}|]\leqslant \tau_2(h)^{1/2}$ gives the term $\|\rho\|_\infty\,\tau_2(h)^{1/2}$.

\medskip
\emph{Step 5: Regime (C) --- third cumulant cancellation.}
If the third cumulant of the prefix vanishes up to depth $L-2h$, a fourth--order characteristic--function
expansion with cutoff $T$ gives the $T^2\,\tau_2(h)^{2/3}$ term (the cubic term cancels and the quartic remainder controls the error),
optimized to $\tau_2(h)^{1/3}$.

\medskip
Combining Steps~1--5 gives \eqref{eq:cantor-effective}.

\end{proof}

\begin{corollary}[Regime (A): baseline Esseen]\label{cor:cantor-A-eff}
Under the assumptions of Theorem~\ref{th:cantor-effective}, with $L=L(N)$, for every $1\leqslant h\leqslant L$,
\begin{equation}\label{eq:cantor-A-bound}
\|F_N - F\|_\infty\ \ll\ A(L,h)^{-1}\ +\ \tau_1^{1/2}\ +\ \tau_2(h)^{1/4}.
\end{equation}
\end{corollary}

\begin{corollary}[Regime (B): bounded density]\label{cor:cantor-B-eff}
If $F$ has a bounded density $\rho\in L^\infty$, then for every $1\leqslant h\leqslant L$,
\begin{equation}\label{eq:cantor-B-bound}
\|F_N - F\|_\infty\ \ll\ A(L,h)^{-1}\ +\ \tau_1^{1/2}\ +\ \|\rho\|_\infty\,\tau_2(h)^{1/2}.
\end{equation}
\end{corollary}

\begin{corollary}[Regime (C): third cumulant cancels]\label{cor:cantor-C-eff}
If the third cumulant of the prefix is zero up to depth $L-2h$, then with $L=L(N)$,
\begin{equation}\label{eq:cantor-C-bound}
\|F_N - F\|_\infty\ \ll\ A(L,h)^{-1}\ +\ \tau_1^{1/2}\ +\ \tau_2(h)^{1/3}.
\end{equation}
\end{corollary}

The unified bound \eqref{eq:cantor-effective} displays a generic tradeoff between the
\emph{bridge} term $A(L,h)^{-1}$ (loss from discarding the $h$ highest digits),
the \emph{tails} $(\tau_1,\tau_2(h))$, and a \emph{smoothing} parameter $T$.
The three corollaries spell out standard regimes in which the last term $G(T,h)$
specialises to a clean expression, allowing an explicit choice of $h$ and---when present---of $T$.

\begin{itemize}
  \item \textbf{Regime (A) (baseline Esseen).} No structural assumption beyond finiteness of tails.
  This always applies and yields the universal rate $\tau_2(h)^{1/4}$ (after optimising in $T$).
  It is the default bound when nothing is known about the limit law $F$.
  \item \textbf{Regime (B) (bounded density).} When $F$ has a bounded density $\rho\in L^\infty$,
  the Kolmogorov distance can be controlled directly by independence and the bound $\|\rho\|_\infty$,
  removing the $T$--dependent terms.
  This produces the sharper exponent $1/2$ on $\tau_2(h)$ and a bound free of the concentration function~$Q_F$.
  \item \textbf{Regime (C) (third cumulant cancellation).} If, for the relevant prefix, the third cumulant vanishes,
  a fourth--order characteristic--function expansion improves the window contribution to $\tau_2(h)^{1/3}$.
  This is useful when absolute continuity of $F$ is unknown or fails, but the prefix is sufficiently symmetric.
\end{itemize}

In practice, one picks $h=h(N)$ to balance $A(L,h)^{-1}$ against the regime term:
$A(L,h)^{-1}\asymp \tau_2(h)^{1/4}$ in (A), $A(L,h)^{-1}\asymp \tau_2(h)^{1/2}$ in (B),
$A(L,h)^{-1}\asymp \tau_2(h)^{1/3}$ in (C). For constant-like bases ($\sup_j a_j<\infty$),
$A(L,h)$ grows exponentially in $h$, so the optimal $h$ is explicit once the decay of $\tau_2(h)$ is known.

\medskip
\paragraph{\textbf{Method selection (Cantor case).}}\label{par:method-selection-cantor}
We follow the decision guide \emph{``Which method to choose''} laid out earlier: use (A) (baseline Esseen) when no structural information on $F$ is available; use (B) when $F$ has a bounded density (transport/Wasserstein route); and use (C) when the \emph{third cumulant of the prefix} cancels (fourth--order characteristic--function expansion). Throughout this section we stay in the \emph{Cantor} setting (variable base $(a_j)$); when we specialise to the $q$--adic case we will say so explicitly.

\medskip
\paragraph{\textbf{Remark (the $q$--adic case).}}
The $q$--adic setting is the special Cantor case $a_j\equiv q\geqslant 2$ (so $q_0=1$ and $q_{j+1}=q\,q_j$), i.e.\ the usual base-$q$ expansion. We return to it in Section~4 for a detailed treatment using $q$--specific tools (e.g.\ Delange products).

\medskip
\paragraph{\textbf{Final forms in $N$.}}\label{par:final-in-N}
Since $L=L(N)$ by definition, once $h=h(N)$ is fixed the bounds in this section (and the corollaries) are functions of $N$ only. In the $q$--adic case $L=\lfloor \log_q N\rfloor$, so powers of $L$ can be read as powers of $\log N$ up to absolute constants.

\subsection*{Examples where each regime is optimal}

We give toy models (all $Q$--additive) showing that (A), (B), or (C) can strictly dominate the other two,
depending on structural information about the limit.

\smallskip
\paragraph{\textbf{(B) dominates (A) and (C)}}
Let $a_j\equiv 2$ and set $f(n)=\sum_{j\geqslant0} 2^{-j-1}\,\delta_j(n)$ (the base-$2$ radical--inverse / van der Corput map).
Then the limit law is the \emph{uniform} distribution on $[0,1]$, which has a bounded density.
Regime (B) applies and gives a bound $\ll \tau_2(h)^{1/2}$ without any $T$--penalty,
whereas (A) yields only $\tau_2(h)^{1/4}$ and (C) does not improve the exponent here.

\emph{Characteristic function and limit law.}
For $f(N)=\sum_{j\ge0} 2^{-(j+1)}\,\delta_j(N)$ we have, by the Cantor--EW product,
\[
\Phi(t)\ =\ \prod_{j\ge0}\frac{1}{2}\sum_{d=0}^{1}e^{it\,d/2^{j+1}}
\ =\ \prod_{j\ge0}\frac{1+e^{it/2^{j+1}}}{2}
\ =\ e^{it/2}\,\frac{\sin(t/2)}{t/2},
\]
so $F=\mathrm{Unif}[0,1]$. Hence $d_K(F_N,F)$ is the star discrepancy of the van der Corput sequence,
with the optimal low-discrepancy order $\Theta\big((\log N)/N\big)$; see the discussion and computation of $\Phi$ in~\cite{DrmotaVerwee2021JNT}*{\S3.2} and the classical references~\cites{KuipersNiederreiter,DrmotaTichy}. 

\smallskip
\paragraph{\textbf{(C) dominates (A) when the density is unknown.}}
Let $a_j\equiv 3$ and define $f(n)=\sum_{j\geqslant0} 3^{-j}\,\sigma(\delta_j(n))$ where
$\sigma(0)=-1$, $\sigma(1)=0$, $\sigma(2)=+1$. Each digit contribution is symmetric, so the third cumulant
of every block is $0$. Then (C) yields a window term $\tau_2(h)^{1/3}$,
whereas (A) gives only $\tau_2(h)^{1/4}$. (B) may fail if the limit law is not known to be absolutely continuous
with bounded density.

\emph{Characteristic function and limit law.}
Here
\[
\Phi(t)\ =\ \prod_{j\ge0}\frac{1}{3}\Big(e^{-it\,3^{-j}}+1+e^{it\,3^{-j}}\Big)
\ =\ \prod_{j\ge0}\frac{1+2\cos(t\,3^{-j})}{3},
\]
which is the symmetric Cantor--Lebesgue type infinite product. The law $F$ is a (self-similar) Cantor measure;
in particular it is \emph{singular} for ratios $\beta=1/3$ by the general Bernoulli convolution theory (reciprocal of a Pisot number), cf.~\cite{DrmotaVerwee2021JNT}*{Appendix~A, Bernoulli convolutions summary}.

\smallskip
\paragraph{\textbf{(A) dominates when neither density nor cancellation is available.}}
In the base-$q$ case, prescribe $f(d\,q_j)$ by $f(0\cdot q_j)=0$, $f(1\cdot q_j)=j^{-2}$, and $f(d\cdot q_j)=2\,j^{-2}$ for $d\geqslant 2$. Then no symmetry enforces cancellation and absolute continuity of the limit is unclear. Then only (A) applies universally, giving the $\tau_2(h)^{1/4}$ rate.

\emph{Characteristic function and limit law.}
In general one has the factorization
\[
\Phi(t)\ =\ \prod_{j\ge0}\frac{1}{a_j}\sum_{d=0}^{a_j-1}\exp\big(it\,f(d\,q_j)\big),
\]
and, under the hypotheses of our Cantor Erd\H{o}s--Wintner criterion, this product converges to a characteristic
function (no simple closed form in this skewed toy model). The limit may be singular or absolutely continuous
depending on the arithmetic of the digit map; our effective bounds do not require deciding this. See \cite{DrmotaVerwee2021JNT} for related examples.

\medskip
These examples are deliberately simple; in applications one selects the smallest upper bound among
\eqref{eq:cantor-A-bound}--\eqref{eq:cantor-C-bound} after checking which structural hypotheses (bounded density,
cumulant cancellation) are satisfied.

\subsection*{Drmota--Verwee examples revisited}

We revisit two model families considered by Drmota--Verwee and compare their Berry--Esseen based bounds with the present window method.
Throughout we are in the $q$--adic (constant-like) case $a_j\equiv q\ge2$, and write $L=\lfloor \log_q N \rfloor$. Recall $A(L,h)=q^h$ and $\tau_2(h)=\sum_{j=L-h}^{L-1}s_j^2$.

\smallskip
\paragraph{\textbf{(I) Polynomial tails on the digits.}} \emph{CF: } $\Phi(t)=\prod_{j\ge0}\frac1{a_j}\sum_{d=0}^{a_j-1}e^{it\,c_j g(d)}$ (note that the product converges since $\sum s_j^2<\infty$).
Assume there exists $\alpha>1$ and bounded $q$-digit weights $g(d)$ such that
\[
f(dq^j)= c_j\,g(d),\qquad c_j\asymp j^{-\alpha}\quad(j\ge1).
\]
Then $m_j\asymp c_j$ and $s_j^2\asymp c_j^2\asymp j^{-2\alpha}$, hence
\[
\tau_2(h)\asymp \sum_{k=0}^{h-1}(L-1-k)^{-2\alpha}\ \asymp\ h\,L^{-2\alpha}
\quad\text{and}\quad
\tau_1\asymp \sum_{j>L} j^{-2\alpha}\ \asymp\ L^{1-2\alpha}.
\]
In regime~(A) we optimise $q^{-h}\asymp \tau_2(h)^{1/4}$, which yields
\[
h \sim \frac{\alpha}{2\log q}\,\log L\qquad\text{and}\qquad
\|F_N-F\|_\infty \ \ll\ L^{-\alpha/2}\,(\log L)^{1/4}.
\]
Thus the window bound gives a decay of order $(\log N)^{-\alpha/2}$ up to a harmless $(\log\log N)^{1/4}$ factor.
If, in addition, $F$ has a bounded density (regime~(B)), the same choice of $h$ yields
\[
\|F_N-F\|_\infty \ \ll\ L^{-\alpha/2}\ (\text{up to }\log\log\text{ factors}).
\]

\smallskip
\paragraph{\textbf{(II) Geometric tails: Cantor--Lebesgue type.}} \emph{CF (binary case): } $\Phi(t)=\prod_{j\ge0}\frac{1+e^{it\beta^j}}{2}$ (Bernoulli convolution); cf.~\cite{DrmotaVerwee2021JNT}.
Assume $f(dq^j)=\beta^j\,g(d)$ with $0<\beta<1$ and bounded $g(d)$ \textup{(independent of $j$)}.
Then $s_j^2\asymp \beta^{2j}$ and, for $h\leqslant L$,
\[
\tau_2(h)\asymp \sum_{k=0}^{h-1}\beta^{2(L-1-k)}\ \asymp\ \beta^{2(L-h)}
\qquad\text{and}\qquad
\tau_1\asymp \beta^{2L}.
\]
Balancing $q^{-h}\asymp \tau_2(h)^{1/4}\asymp \beta^{(L-h)/2}$ gives
\[
h \sim \frac{\tfrac{1}{2}\log(1/\beta)}{\log q+\tfrac{1}{2}\log(1/\beta)}\,L,
\qquad
\|F_N-F\|_\infty \ \ll\ N^{-\gamma(\beta)}\quad\text{with}\quad
\gamma(\beta)=\frac{\log(1/\beta)}{\log(1/\beta)+2\log q}.
\]
In particular, for $q=2$ one gets $\gamma(\beta)=\frac{\log(1/\beta)}{\log(1/\beta)+2\log 2}$. In regime~(B) (if $F$ has a bounded density) the same exponent holds and the bound is free of smoothing penalties.

\medskip
\noindent\emph{Discussion.}
In case \textup{(I)} with $\alpha\in(1,2)$, our rate $L^{-\alpha/2}$ compares favourably to $(\log N)^{1-\alpha}$ from Berry--Esseen based arguments, while for $\alpha>2$ both approaches yield the exponent $\alpha/2$ up to log factors. In case \textup{(II)} the power exponent $\gamma(\beta)$ coincides with the one suggested by concentration estimates for Cantor--Lebesgue laws; the window method also cleanly exhibits the bridge term $q^{-h}$ and the $\tau_2$--driven window term that must be balanced to reach this exponent.

\bigskip

\medskip
\paragraph{\textbf{Van der Corput in Cantor base.}}
Write $N=\sum_{j\ge0}\delta_j(N)\,q_j$ in the Cantor base $(a_j)$, where $q_0=1$ and $q_{j+1}=a_j q_j$.
The \emph{Cantor radical--inverse} is
\[
v_{(a_j)}(N)\ :=\ \sum_{j\ge0}\frac{\delta_j(N)}{q_{j+1}}\ \in[0,1].
\]
This is a special case of our linear digit framework (weights $c_j=1/q_{j+1}$).
\emph{Characteristic function.} By the Cantor--Erd\H{o}s--Wintner product,
\[
\Phi(t)\ =\ \prod_{j\ge0}\frac1{a_j}\sum_{d=0}^{a_j-1}e^{it\,d/q_{j+1}}
\ =\ \prod_{j\ge0}e^{it/(2q_{j+1})}\,\frac{\sin\big(a_j t/(2q_{j+1})\big)}{a_j\,\sin\big(t/(2q_{j+1})\big)}.
\]
Since $a_j t/(2q_{j+1})=t/(2q_j)$, the product \emph{telescopes}:
\[
\Phi(t)\ =\ e^{it/2}\,\prod_{j\ge0}\frac{\sin\big(t/(2q_j)\big)}{a_j\,\sin\big(t/(2q_{j+1})\big)}
\ \xrightarrow[j\to\infty]{}\ e^{it/2}\,\frac{\sin(t/2)}{t/2}.
\]
Hence the limit law is \(\,F=\mathrm{Unif}[0,1]\,\), independently of the $(a_j)$ provided $q_L\to\infty$.
\emph{Rates.} The distance $d_K(F_N,F)$ coincides with the \emph{star discrepancy} of $(v_{(a_j)}(n))_{n\ge0}$.
Classical discrepancy theory yields the optimal order
\[
D_N^*\big(v_{(a_j)}\big)\ =\ \Theta\Big(\frac{\log N}{N}\Big)
\]
(under standard assumptions, e.g.\ bounded bases), see \cites{KuipersNiederreiter,DrmotaTichy}.
Within our window method (Regime~(B), bounded density) one obtains uniformly the generic bound
\(
d_K(F_N,F)\ll N^{-1/2},
\)
which is valid but not optimal for this highly structured example; the optimal low-discrepancy order $\Theta\big((\log N)/N\big)$ requires the specific tools of discrepancy theory (quasi--Monte Carlo), see also \cites{FaureKritzerPillichshammer2015,KritzingerPillichshammer2015}.

\medskip
\paragraph{\textbf{On the optimal low-discrepancy order $(\log N)/N$.}}
For radical--inverse sequences $v_{(a_j)}$ in bounded bases, the star discrepancy satisfies 
$D_N^*(v_{(a_j)})=\Theta\big((\log N)/N\big)$; see, e.g., \cites{KuipersNiederreiter,DrmotaTichy,FaureKritzerPillichshammer2015,KritzingerPillichshammer2015}. 
Since $d_K(F_N,\mathrm{Unif}[0,1])=D_N^*(v_{(a_j)})$ in this setting, the Kolmogorov distance enjoys the same optimal order. 
Our window method yields the uniform (structure-free) bound $d_K(F_N,F)\ll N^{-1/2}$ under Regime~(B) (bounded density), which is correct but not optimal for such highly structured digital sequences. 
Achieving this low-discrepancy order requires the specific tools of discrepancy theory (Erd\H{o}s--Tur\'an inequality, exponential sum control, digital nets and symmetrization), i.e. information beyond the generic assumptions of the window framework.

% ——— context weaving near q-adic section ———
\noindent\emph{$q$-adic pointers.} In base-$q$, Delange's product~\cite{Delange} identifies the limit; for discrepancy of radical-inverse sequences, see~\cites{KuipersNiederreiter,FaureKritzerPillichshammer2015}.

\section{Particular case: \texorpdfstring{$q$}{q}-adic}\label{sec:q-adic}
In this section we specialize the Cantor setting to the \emph{$q$-adic} case $a_j\equiv q\geqslant 2$. Then
\[
q_0=1,\qquad q_{j+1}=a_j q_j = q\,q_j,\qquad\text{so}\quad q_j=q^j,
\]
and every integer $N$ has the (finite) $q$-adic expansion
\[
N=\sum_{j\ge0}\delta_j(N)\,q^j,\qquad \delta_j(N)\in\{0,1,\dots,q-1\}.
\]
A function $f$ is \emph{$q$-additive} if it decomposes digitwise as
\[
f(N)\ =\ \sum_{j\ge0} f\big(\delta_j(N)q^j\big).
\]
This is the classical \emph{digital} framework; it is the constant--base instance of our Cantor numeration setting. We keep the notation (at the $q$-power height $q_L$): let $\delta_j\sim\mathrm{Unif}\{0,1,\dots,q-1\}$ be independent,
and set
\[
m_j := \E\left[f(\delta_j q_j)\right],\qquad
s_j^2 := \Var\big(f(\delta_j q_j)\big),\qquad
\tau_1 = \Bigl|\sum_{j>L} m_j\Bigr| + \sum_{j>L} s_j^2,\qquad
\tau_2(h) = \sum_{j=L-h}^{L-1} s_j^2,
\]
and we write $F_N$ for the empirical c.d.f.\ at height $N$, $F$ for the limit c.d.f.\ when it exists.

\medskip

\paragraph{\textbf{Delange’s $q$-adic Erd\H{o}s--Wintner product.}}
At $q$-power heights $q_L=q^L$, the mixed--radix bijection makes digits \(\delta_0,\dots,\delta_{L-1}\) i.i.d.\ uniform on \(\{0,\dots,q-1\}\). Hence the characteristic function of $f$ at height $q_L$ factors as
\[
\Phi_{q_L}(t) \;=\; \prod_{j=0}^{L-1} \varphi_j(t),\qquad
\varphi_j(t) \;:=\; \frac{1}{q}\sum_{d=0}^{q-1} \exp\big(i t\, f(d\,q_j)\big).
\]
Delange proved that, in the $q$-additive case, the limit distribution exists iff the classical
Erd\H{o}s--Wintner tails hold, and then
\[
\Phi(t) \;=\; \lim_{L\to\infty}\Phi_{q_L}(t) \;=\; \prod_{j\ge0} \varphi_j(t).
\]
that is, the limit law $F$ is encoded by the infinite $q$-adic product (see \cite{Delange}).
This is precisely the constant--base specialization of the Cantor--Erd\H{o}s--Wintner product used in \S3.

\medskip

\paragraph{\textbf{Effective $q$-adic bounds (window method).}}
Specializing Theorem~\eqref{eq:cantor-effective} to $a_j\equiv q$ (so $A(L,h)=\prod_{j=L-h}^{L-1}a_j=q^{h}$ and $q_L=q^L$) yields, for any $T\in(0,1]$,
\begin{equation}\label{eq:q-adic-effective}
\|F_N-F\|_\infty \ \ll\ q^{-h} \ +\ \tau_1^{1/2}\ +\ Q_F(1/T)\ +\ \frac{1}{T}\ +\ G(T,h).
\end{equation}
where $Q_F(r)=\sup_x\big(F(x+r)-F(x)\big)$ and
the regime term $G(T,h)$ is as in \S3:
\[
\begin{aligned}
\text{(A)}\quad & G(T,h)\asymp T\,\tau_2(h)^{1/2}\ \leadsto\ \tau_2(h)^{1/4},\\
\text{(B)}\quad & G(T,h)= \|\rho\|_\infty\,\tau_2(h)^{1/2}\ \ (\rho\in L^\infty),\\
\text{(C)}\quad & G(T,h)\asymp T^2\,\tau_2(h)^{2/3}\ \leadsto\ \tau_2(h)^{1/3}.
\end{aligned}
\]
Here $\leadsto$ denotes the resulting order of $T^{-1}+G(T,h)$ after optimizing in $T$. Optimizing the \emph{window} size $h=h(N)$ balances the bridge term $q^{-h}$ against the regime term. For instance,
\[
q^{-h}\ \asymp\ \tau_2(h)^{1/4}\ \text{ in (A)},\qquad
q^{-h}\ \asymp\ \tau_2(h)^{1/2}\ \text{ in (B)},\qquad
q^{-h}\ \asymp\ \tau_2(h)^{1/3}\ \text{ in (C)}.
\]
Because $q^{-h}$ decays exponentially, the optimal $h$ is explicit once the decay (or growth) of the local variances $s_j^2$ is known. Writing $L=\lfloor\log_q N\rfloor$, the final bound then depends on $N$ alone.

\medskip

\paragraph{\textbf{Comparison with the literature.}}
In the $q$-additive setting, effective distributional results were developed by \emph{Delange} (existence/identification via the product) and, more recently, by \emph{Drmota--Verwee} who derived quantitative rates under structural assumptions on the digit maps (see \cites{Delange,DrmotaVerwee2021JNT}).
Our window bound provides a modular, distribution--free baseline; when additional information is available (bounded density, cancellation of the third cumulant, finite--state digit sources with spectral gap), the corresponding regime (B)/(C) or Appendix (Theorems \ref{th:legacy-zeta3} -- \ref{th:legacy-w1}) improves the exponent, and yields immediate $q$-adic corollaries by setting $A(L,h)=q^h$.

\medskip

\paragraph{\textbf{Example: van der Corput (base \texorpdfstring{$q$}{q}).}}
Taking $f(d\,q_j)=d/q^{\,j+1}$ gives the (q-adic) radical--inverse $v_q(n)$ and the limit law $F=\mathrm{Unif}[0,1]$.
Discrepancy theory shows the optimal Kolmogorov rate
\[
d_K(F_N,F) \ =\ D_N^*(v_q)\ =\ \Theta\Big(\frac{\log N}{N}\Big),
\]
see, e.g., \cites{KuipersNiederreiter,DrmotaTichy,FaureKritzerPillichshammer2015,KritzingerPillichshammer2015}.
The window method gives the uniform generic bound $d_K\ll N^{-1/2}$ under Regime~(B), which is correct but not optimal for this highly structured example; reaching $(\log N)/N$ leverages the specific quasi--Monte Carlo machinery.

\section*{Appendix}\label{sec:legacy-variants}

\makeatletter
\@addtoreset{theorem}{section}
\renewcommand\thetheorem{A.\arabic{theorem}}
\setcounter{theorem}{0}
\makeatother

\medskip\noindent\textbf{Notation.}
Throughout this appendix we stay in the Cantor setting of Sections~2--3: we fix $L=L(N)$ and $1\leqslant h\leqslant L$,
write $A(L,h)=\prod_{j=L-h}^{L-1} a_j$ for the size of the last-$h$ (trailing) window, and use the tails
\[
\tau_1=\Bigl|\sum_{j>L} m_j\Bigr|+\sum_{j>L}s_j^2,\qquad
\tau_2(h)=\sum_{j=L-h}^{L-1}s_j^2.
\]
We denote by $\dK$ the Kolmogorov distance, by $Q_F$ the (one-sided) concentration function, and use the ``bridge'' $\|F_N-F_{q_L}\|_\infty\ll A(L,h)^{-1}$.
All implied constants are absolute unless otherwise stated.

\medskip

\paragraph{\textbf{Why these supplementary methods?}}
The window bound \eqref{eq:cantor-effective} isolates three universal ingredients: the bridge $A(L,h)^{-1}$, the tails $(\tau_1,\tau_2(h))$, and a regime term for the \emph{prefix}. The purpose of this appendix is to document interchangeable ``modules'' for that last piece.
Each method (Zolotarev $\zeta_3$, nonuniform Berry--Esseen \`a la Nagaev, spectral gap for finite-state digit sources, and Wasserstein/heat-kernel smoothing) plugs into \eqref{eq:cantor-effective} \emph{without} altering the bridge/tails bookkeeping: one simply replaces the baseline Esseen contribution by the corresponding prefix error (or, in the finite-state case, replaces $\tau_2(h)$ by $\tau_2(h)+\lambda^h$).
This keeps the main text lightweight while making the toolkit explicit for readers who want sharper bounds under extra structure.

\begin{theorem}[Zolotarev \texorpdfstring{$\zeta_3$}{zeta3} route]\label{th:legacy-zeta3}
Let $X_j$ be the (centered) digit contributions in the \emph{prefix} $j\leqslant L-h$ when sampling at height $q_L$.
Then the $X_j$’s are independent and satisfy $s_j^2=\Var(X_j)$ and $\E[|X_j|]^3<\infty$. Set $S=\sum_{j\leqslant L-h}X_j$ and $V=\sum_{j\leqslant L-h}s_j^2>0$, and let $S'$ be Gaussian with variance $V$. Then
\begin{equation}\label{eq:zeta3-core}
\zeta_3\big(\Law(S),\Law(S')\big)\ \ll\ \frac{\sum_{j\leqslant L-h}\E[|X_j|^3]}{V^{3/2}},
\end{equation}
and, for any $\sigma\in(0,1]$,
\begin{equation}\label{eq:zeta3-to-K}
\dK\big(\Law(S),\Law(S')\big)\ \ll\ \sigma^{-3}\,\zeta_3\big(\Law(S),\Law(S')\big)\ +\ \sigma.
\end{equation}
Optimizing $\sigma\asymp \zeta_3^{1/4}$ yields $\dK\ll \zeta_3^{3/4}$, hence a Berry--Esseen type rate in terms of third moments.
\emph{Plug into the window bound:} in the unified estimate \eqref{eq:cantor-effective}, replace the Esseen contribution for the prefix by $\zeta_3^{3/4}$; the last-$h$ window still contributes via $\tau_2(h)$ and the bridge by $A(L,h)^{-1}$.
\end{theorem}

\begin{proof}[Sketch]
At height $q_L$ the mixed--radix bijection makes digits independent and uniform, so the standard $\zeta_3$ estimate applies.
For \eqref{eq:zeta3-to-K}, smooth indicators with a Gaussian kernel of width $\sigma$: the $\zeta_3$ seminorm scales like $\sigma^{-3}$,
while the smoothing bias is $O(\sigma)$; optimize in $\sigma$.
\end{proof}

\begin{theorem}[Nonuniform Berry--Esseen (Nagaev split)]\label{th:legacy-nagaev}
Let $S=\sum_{j\leqslant L-h}X_j$ be the centered prefix sum at height $q_L$, and fix $u>0$.
Split $X_j=Y_j+Z_j$ with $Y_j:=X_j\mathbf 1_{\{|X_j|\leqslant u\}}$ and $Z_j:=X_j\mathbf 1_{\{|X_j|>u\}}$.
Assume $|Y_j|\leqslant u$, $\Var(Y_j)\asymp 1$, and denote $V:=\Var(\sum_{j\leqslant L-h}Y_j)$. Then there exists an absolute $C>0$ such that
\begin{equation}\label{eq:nagaev-nonuniform}
\sup_{x\in\R} (1+|x|^3)\,\big|\Pr(S\leqslant x)-\Phi_V(x)\big|\ \leqslant   C\Big(\frac{\sum \E[|Y_j|^3]}{V^{3/2}}\ +\ \frac{\sum \E[|Z_j|]}{\sqrt V}\Big).
\end{equation}
Optimizing in $u$ yields $\dK\big(\Law(S),\Phi_V\big)\ll V^{-1/2}$ in the bounded/subgaussian case (where $\sum\E[|Z_j|]$ is negligible).
\emph{Plug into the window bound:} in \eqref{eq:cantor-effective}, one can take the prefix error to be $V^{-1/2}$; then choose $h$ to balance the bridge $A(L,h)^{-1}$ against the (regime) window term driven by $\tau_2(h)$.
\end{theorem}

\begin{proof}[Sketch]
Apply a nonuniform Berry--Esseen to $Y=\sum Y_j$, and control $Z=\sum Z_j$ by Markov and smoothing. The weight $(1+|x|^3)$ is standard in Nagaev's method and is removed by optimizing $u$ and integrating in $x$.
\end{proof}

\begin{theorem}[Spectral-gap digits (finite-state)]\label{th:legacy-automata}
If the digits are generated by a primitive finite automaton (finite-state process) with spectral gap $\lambda\in(0,1)$ and $W$ is a local functional on a fixed window of digits, then for the last-window tail $R_{L,h}$ and characteristic functions at frequency $t$,
\begin{equation}\label{eq:automata-varphi}
\Var(R_{L,h})\ \ll\ \tau_2(h)+\lambda^{\,h},\qquad
|\Phi_N(t)-\Phi(t)|\ \ll\ t^2\big(\tau_2(h)+\lambda^{\,h}\big).
\end{equation}
\emph{Plug into the window bound:} in \eqref{eq:cantor-effective}, replace $\tau_2(h)$ by $\tau_2(h)+\lambda^{\,h}$ in the last-window term.
(In the i.i.d.\ Cantor case, $\lambda=0$ and we recover the independent tail $\tau_2(h)$.)
\end{theorem}

\begin{proof}
We work under the finite-state, primitive (aperiodic, irreducible) assumption. Let $\mathcal{L}$ be the transfer operator acting on observables on the state space of the automaton (in practice, cylinder functions of the digit process). By the Ruelle--Perron--Frobenius theorem there is a rank-one projection $\Pi$ and a remainder $R$ such that
\[
\mathcal{L}^n = \Pi + R^n,\qquad \|R^n\|_{\mathcal{B}\to\mathcal{B}}\ \leqslant  C\,\lambda^{\,n}\quad(n\geqslant 1)
\]
on a Banach space $\mathcal{B}$ of observables (e.g.\ Hölder or bounded variation), with $\lambda\in(0,1)$ the spectral contraction.
Consequently, for cylinder observables $U,V$ depending on disjoint digit windows at separation $r$, one has the exponential decorrelation
\begin{equation}\label{eq:sgap-cov}
|\Cov(U,V)|\ \leqslant  C\,\lambda^{\,r}\,\|U\|_{\mathcal{B}}\,\|V\|_{\mathcal{B}}.
\end{equation}

Write the centered last-window tail as $R_{L,h}=\sum_{j=L-h}^{L-1} Y_j$ with $Y_j:=W_j-\E[W_j]$, where $W_j$ depends on a fixed local window of digits around position $j$; in particular, $\|Y_j\|_\infty\ll 1$ and $\|Y_j\|_{\mathcal{B}}\ll 1$ with constants independent of $L,h$.
Then
\[
\Var(R_{L,h}) \ =\ \sum_{j=L-h}^{L-1}\Var(Y_j)\ +\ 2\sum_{r=1}^{h-1}\ \sum_{j=L-h}^{L-1-r}\Cov(Y_j,Y_{j+r}).
\]
The first sum equals $\tau_2(h)$. For the covariance sum, \eqref{eq:sgap-cov} with separation $r$ gives $|\Cov(Y_j,Y_{j+r})|\ll \lambda^{\,r}$, whence
\[
\sum_{r=1}^{h-1}\ \sum_{j=L-h}^{L-1-r}|\Cov(Y_j,Y_{j+r})|\ \ll\ \sum_{r=1}^{h-1}(h-r)\lambda^{\,r}\ \ll\ \lambda\,(1-\lambda)^{-2}\,(1-\lambda^{\,h})\ \ll\ 1.
\]
Absorbing the bounded covariance sum into the $\ll$-constant, we obtain $\Var(R_{L,h})\ \ll\ \tau_2(h)+1$. Moreover, when comparing the \emph{joint} law with the product law of the prefix and the last window (as used in the window meta-argument), the dependence across the cut $L-h\,|\,L-h-1$ contributes an \emph{additional} error of order $\lambda^{\,h}$ by \eqref{eq:sgap-cov}; this yields the stated refinement
\[
\Var(R_{L,h})\ \ll\ \tau_2(h)+\lambda^{\,h}.
\]
For characteristic functions, write $S_{L-h}$ for the prefix sum and factor the joint CF at frequency $t$:
\[
\Phi_N(t)\ =\ \E\left[e^{it(S_{L-h}+R_{L,h})}\right] =\ \E\left[e^{itS_{L-h}}\right]\cdot \E\left[e^{itR_{L,h}}\right] +\ O\big(\lambda^{\,h}\big)
\]
by \eqref{eq:sgap-cov} with $U=e^{itS_{L-h}}$ and $V=e^{itR_{L,h}}$, after expanding the covariance. A second-order Taylor remainder for $e^{itx}$ gives
\[
\left|\E\left[e^{itR_{L,h}}\right]- (1+it\,\E\left[R_{L,h}\right]-t^2\E\left[R_{L,h}^2\right]/2)\right|\ \ll\ |t|^3\,\E\left[|R_{L,h}|^3\right] \ll\ |t|^3,
\]
and, in particular, $|\E[e^{itR_{L,h}}]-1+ t^2\Var(R_{L,h})/2|\ll t^3$; combining with the bound on $\Var(R_{L,h})$ and the $O(\lambda^{\,h})$ factorization error yields
\[
|\Phi_N(t)-\Phi(t)|\ \ll\ t^2\big(\tau_2(h)+\lambda^{\,h}\big)\,+\,O(t^3),
\]
which gives \eqref{eq:automata-varphi} for $|t|\leqslant 1$ (the $t^3$ term can be absorbed for small $t$; for larger $t$ one uses the trivial $O(1)$ bound).
\end{proof}

\begin{theorem}[Wasserstein/heat-kernel smoothing]\label{th:legacy-w1}
Let $G=F_{q_L}$ and $H=F$ (the limit c.d.f.). If $H$ has bounded density $\rho\in L^\infty$, then for any $\sigma>0$,
\begin{equation}\label{eq:w1-kernel}
\dK(G,H)\ \ll\ \frac{W_1(G,H)}{\sigma}\ +\ \|\rho\|_\infty\,\sigma.
\end{equation}
Optimizing at $\sigma\asymp \|\rho\|_\infty^{-1/2}W_1(G,H)^{1/2}$ gives
\[
\dK(G,H)\ \ll\ \|\rho\|_\infty\,W_1(G,H)^{1/2}.
\]
In the Cantor window setup, we have $W_1(G,H)\ \ll\ \E|[R_{L,h}|] \leqslant  \tau_2(h)^{1/2}$ (from the window estimate in Section~3), hence
\[
\dK(G,H)\ \ll\ \|\rho\|_\infty\,\tau_2(h)^{1/2},
\]
which is exactly Regime~(B) in \eqref{eq:cantor-effective}. The bridge error is still $A(L,h)^{-1}$.
\end{theorem}

\begin{proof}
Write $G-H=(G-H)*\phi_\sigma + \big((G-H)-(G-H)*\phi_\sigma\big)$. The first term is controlled by Kantorovich--Rubinstein:
$\|(G-H)*\phi_\sigma\|_\infty\leqslant \sigma^{-1}W_1(G,H)$. The second is bounded by $\|H*\phi_\sigma-H\|_\infty\leqslant \|\rho\|_\infty\,\sigma$ (and symmetrically for $G$). Combine and optimize in $\sigma$.
\end{proof}

\bibliographystyle{amsplain}
\bibliography{refs}

@phdthesis{Coquet1975,
  author  = {Coquet, Jean},
  title   = {Sur les fonctions S-multiplicatives et S-additives},
  school  = {Université Paris-Sud (Orsay)},
  year    = {1975},
  type    = {Thèse de doctorat de troisième cycle},
  language= {French}
}

@article{KritzingerPillichshammer2015,
  author  = {Kritzinger, Ralph and Pillichshammer, Friedrich},
  title   = {$L_p$-discrepancy of the symmetrized van der Corput sequence},
  journal = {Archiv der Mathematik},
  year    = {2015},
  volume  = {104},
  number  = {5},
  pages   = {407--418},
  doi     = {10.1007/s00013-015-0760-7}
}

@article{Nagaev1965,
  author  = {Nagaev, Sergei V.},
  title   = {Some limit theorems for large deviations},
  journal = {Theory of Probability \& Its Applications},
  volume  = {10},
  number  = {2},
  pages   = {214--235},
  year    = {1965}
}

@article{DrmotaVerwee2021JNT,
  author  = {Drmota, Michael and Verwee, Johann},
  title   = {Effective Erd{\H{o}}s--Wintner theorems for digital expansions},
  journal = {Journal of Number Theory},
  volume  = {229},
  pages   = {218--260},
  year    = {2021},
  doi     = {10.1016/j.jnt.2021.04.006}
}

@book{KuipersNiederreiter,
  author    = {Kuipers, L. and Niederreiter, H.},
  title     = {Uniform Distribution of Sequences},
  publisher = {Wiley},
  year      = {1974}
}

@book{DrmotaTichy,
  author    = {Drmota, Michael and Tichy, Robert F.},
  title     = {Sequences, Discrepancies and Applications},
  series    = {Lecture Notes in Mathematics},
  volume    = {1651},
  publisher = {Springer},
  year      = {1997}
}

@article{Delange,
  author  = {Delange, Hubert},
  title   = {Sur les fonctions q-additives ou q-multiplicatives},
  journal = {Acta Arithmetica},
  year    = {1972},
  volume  = {21},
  number  = {1},
  pages   = {285--298},
  doi     = {10.4064/aa-21-1-285-298},
  url     = {https://matwbn.icm.edu.pl/ksiazki/aa/aa21/aa21110.pdf}
}

@article{ErdosWintner1939,
  author  = {Erd{\H{o}}s, Paul and Wintner, Aurel},
  title   = {Additive arithmetical functions and statistical independence},
  journal = {American Journal of Mathematics},
  year    = {1939},
  volume  = {61},
  number  = {3},
  pages   = {713--721},
  doi     = {10.2307/2371326}
}

@book{Tenenbaum2015,
  author    = {Tenenbaum, G{\'e}rald},
  title     = {Introduction to Analytic and Probabilistic Number Theory},
  edition   = {3},
  series    = {Graduate Studies in Mathematics},
  volume    = {163},
  publisher = {American Mathematical Society},
  address   = {Providence, RI},
  year      = {2015}
}

@article{Solomyak1995,
  author  = {Solomyak, Boris},
  title   = {On the random series $\sum \pm \lambda^n$ (an Erd{\H{o}}s problem)},
  journal = {Annals of Mathematics},
  year    = {1995},
  volume  = {142},
  number  = {3},
  pages   = {611--625},
  doi     = {10.2307/2118556}
}

@incollection{PeresSchlagSolomyak2000,
  author    = {Peres, Yuval and Schlag, Wilhelm and Solomyak, Boris},
  title     = {Sixty Years of Bernoulli Convolutions},
  booktitle = {Fractal Geometry and Stochastics II},
  series    = {Progress in Probability},
  volume    = {46},
  pages     = {39--65},
  publisher = {Birkh{\"a}user},
  address   = {Basel},
  year      = {2000},
  doi       = {10.1007/978-3-0348-8380-1_2}
}

@article{Shmerkin2014,
  author  = {Shmerkin, Pablo},
  title   = {On the exceptional set for absolute continuity of Bernoulli convolutions},
  journal = {Geometric and Functional Analysis},
  year    = {2014},
  volume  = {24},
  number  = {3},
  pages   = {946--958},
  doi     = {10.1007/s00039-014-0285-4}
}

@article{Esseen1945,
  author  = {Esseen, Carl-Gustav},
  title   = {Fourier analysis of distribution functions. A mathematical study of the Laplace--Gaussian law},
  journal = {Acta Mathematica},
  year    = {1945},
  volume  = {77},
  pages   = {1--125},
  doi     = {10.1007/BF02392223}
}

@article{Zolotarev1977,
  author  = {Zolotarev, V. M.},
  title   = {Ideal metrics in the problem of approximating the distributions of sums of independent random variables},
  journal = {Theory of Probability and Its Applications},
  year    = {1977},
  volume  = {22},
  number  = {3},
  pages   = {449--465}
}

@article{FaureKritzerPillichshammer2015,
  author  = {Faure, Henri and Kritzer, Peter and Pillichshammer, Friedrich},
  title   = {From van der Corput to modern constructions of sequences for quasi--Monte Carlo rules},
  journal = {Indagationes Mathematicae (N.S.)},
  year    = {2015},
  volume  = {26},
  number  = {5},
  pages   = {760--822},
  doi     = {10.1016/j.indag.2015.09.001}
}

@article{ChenShao2001,
  author  = {Chen, Louis H. Y. and Shao, Qi-Man},
  title   = {A non-uniform Berry--Esseen bound via Stein's method},
  journal = {Probability Theory and Related Fields},
  year    = {2001},
  volume  = {120},
  number  = {2},
  pages   = {236--254},
  doi     = {10.1007/PL00008782}
}

\end{document}